\documentclass{article}
\pdfoutput=1

\usepackage[a4paper]{geometry}

\usepackage{pdfsync}
\usepackage[hidelinks,hypertexnames=false,hyperindex=true,pdfpagelabels,linktoc=all]{hyperref}
\usepackage{latexsym,amsfonts,amsmath,amssymb}
\usepackage{graphicx}

\usepackage[mathscr]{euscript}

\usepackage[usenames,dvipsnames,svgnames]{xcolor}

\usepackage{amsthm}

\newtheorem{theorem}{Theorem}[section]
\newtheorem{lemma}[theorem]{Lemma}
\newtheorem{algorithm}[theorem]{Algorithm}

\theoremstyle{definition}

\numberwithin{equation}{section}

\newcommand{\satop}[2]{\stackrel{\scriptstyle{#1}}{\scriptstyle{#2}}}

\newcommand{\bsgamma}{{\boldsymbol{\gamma}}}

\newcommand{\bsh}{{\boldsymbol{h}}}

\newcommand{\bsell}{{\boldsymbol{\ell}}}

\newcommand{\bsx}{{\boldsymbol{x}}}
\newcommand{\bsq}{{\boldsymbol{q}}}

\newcommand{\bsz}{{\boldsymbol{z}}}

\newcommand{\bszero}{{\boldsymbol{0}}}

\newcommand{\rd}{\mathrm{d}}
\newcommand{\ri}{\mathrm{i}}

\newcommand{\bbZ}{\mathbb{Z}}
\newcommand{\bbN}{\mathbb{N}}

\newcommand{\calA}{\mathcal{A}}

\newcommand{\calO}{\mathcal{O}}

\newcommand{\calT}{\mathcal{T}}
\newcommand{\scrP}{\mathscr{P}}
\newcommand{\scrQ}{\mathscr{Q}}

\newcommand{\mask}[1]{}

\newcommand{\supp}{{\mathrm{supp}}}
\newcommand{\setu}{{\mathfrak{u}}}

\newcommand{\setw}{{\mathfrak{w}}}

\begin{document}

\title
{Lattice Algorithms for Multivariate Approximation
\\
in Periodic Spaces with General Weight Parameters
}

\author{
    Ronald Cools\footnote{Department of Computer Science, KU Leuven,
        Celestijnenlaan 200A, 3001 Leuven, Belgium,
        \texttt{(ronald.cools|dirk.nuyens)@cs.kuleuven.be}}
\and
    Frances Y.~Kuo\footnote{School of Mathematics and Statistics,
        University of New South Wales, Sydney NSW 2052, Australia,
        \texttt{(f.kuo|i.sloan)@unsw.edu.au}}
\and
    Dirk Nuyens\footnotemark[1]
\and
    Ian H.~Sloan\footnotemark[2]
}

\date{October 2019}
\maketitle

\begin{abstract}
This paper provides the theoretical foundation for the construction of
lattice algorithms for multivariate $L_2$ approximation in the worst case
setting, for functions in a periodic space with general weight parameters.
Our construction leads to an error bound that achieves the best possible
rate of convergence for lattice algorithms. This work is motivated by PDE
applications in which bounds on the norm of the functions to be
approximated require special forms of weight parameters (so-called POD
weights or SPOD weights), as opposed to the simple product weights covered
by the existing literature. Our result can be applied to other
lattice-based approximation algorithms, including kernel methods or
splines.
\\[2mm]
\end{abstract}

\maketitle

\section{Introduction}

This paper provides a theoretical foundation for the construction of
lattice algorithms for multivariate $L_2$ approximation in the worst case
setting, for functions in a periodic space with general weight parameters.
Our construction leads to an error bound that achieves the best possible
rate of convergence for lattice algorithms. We will provide a background
in the Introduction, assuming little prior knowledge from the reader, and
highlight our new contribution together with our motivation for this work.
Section~\ref{sec:form} provides the mathematical formulation of the
problem and reviews known results, while Section~\ref{sec:cbc} proves the
main theorem.

\emph{Lattice rules} have been developed since the late 1950s as cubature
rules for multivariate periodic integrands characterized by absolutely
convergent Fourier series. In recent years lattice rules have also been
successfully used for non-periodic integrands (by way of random shifts or
tent transformation). Lattice rules represent a branch of the family of
\emph{quasi-Monte Carlo \textnormal{(}QMC\textnormal{)} methods}. The
other significant branch of QMC methods encompasses \emph{digital nets and
sequences}. Reference books and surveys include
\cite{Nie92,SJ94,Hic98b,HH02,CN08,Lem09,DP10,LM12,DKS13,LP14,Nuy14}. Our
interest lies in the situations where the dimensionality, $d$, or the
number of variables, is very large, say, in the hundreds or thousands.
Much of the research focus in the last two decades has been on the concept
of \emph{strong tractability} \cite{NW08,NW10,NW12}: loosely speaking, it
means seeking error bounds that are independent of dimension $d$ (or, in
the case of \emph{polynomial tractability}, with error bounds that grow
only polynomially as $d$ increases).

By now it is well known that these desired \emph{dimension-independent
error bounds} hold for carefully chosen lattice rules and suitably defined
\emph{weighted function space} for the integrands \cite{DKS13}. The first
studied setting involves the so-called \emph{product weights}
\cite{SW98,SW01}, where one weight parameter $\gamma_j>0$ is associated
with each coordinate direction $x_j$ to describe the significance of the
integrand $f(\bsx) = f(x_1,\ldots,x_d)$ in the direction $x_j$. The
typical condition to ensure dimension independence is that the sequence
$\{\gamma_j\}$ decays fast enough to ensure that the series is summable,
i.e., $\sum_{j\ge 1} \gamma_j < \infty$. Faster rates of decay of the
weights $\gamma_j$ then enable faster rates of convergence in the
dimension-independent error bounds, provided that the integrands are
sufficiently smooth and the cubature rules are capable of benefitting from
the higher smoothness.

The \emph{component-by-component} (\emph{CBC}) \emph{construction} of
lattice rules that can achieve dimension-independent cubature error bounds
in weighted spaces is another milestone in the past 20 years
\cite{SR02,DKS13}. These constructions are proved to achieve the
\emph{optimal convergence rates} \cite{Kuo03}, while \emph{fast CBC
algorithms} (based on the fast Fourier transform) allow these
constructions to easily reach tens of thousands of dimensions with
millions of points \cite{NC06a,Nuy14}. For example, in the periodic
Hilbert space setting where the squared Fourier coefficients decay at the
rate of $\alpha>1$ (corresponding roughly to $\alpha/2$ available mixed
derivatives), the optimal convergence rate is
$\calO(n^{-\alpha/2+\delta})$, $\delta>0$, where the implied constant is
independent of $d$ provided that $\sum_{j\ge 1}
\gamma_j^{1/\alpha}<\infty$ for the case of product weights, while the
cost for a fast CBC construction with $n$ points up to dimension $d$ is
$\calO(d\,n\log(n))$ operations.

Lattice rules have also been analyzed in the context of multivariate
approximation. We refer to the resulting algorithms as \emph{lattice
algorithms}; they can be described as follows. For a function with an
absolutely convergent Fourier series, we approximate this function by
first truncating the series expansion to a finite index set, and then
approximating the remaining Fourier coefficients (which are integrals of
the function against each basis function) by lattice rules. Lattice
algorithms have been considered in a number of settings. For example, in
the worst case $L_2$ setting considered in this paper, we measure the
approximation error in the $L_2$ norm, and consider the largest possible
error for functions over the unit ball of our function space. Two main
strategies have been employed in the literature: one strategy is to
construct lattice algorithms to directly minimize the error bound
\cite{KSW06,KSW08}; the other strategy is to construct lattice algorithms
which exactly reconstruct the function on a given finite index set (the
latter are called \emph{reconstruction lattices})
\cite{Kam13,KPV15,PV16,BKUV17,KMNN}. Both strategies can make use of CBC
constructions. It is also possible to combine both strategies in one CBC
construction. {Also related are \emph{spline algorithms} or \emph{kernel
methods} \cite{ZLH06,ZKH09} and collocation \cite{LH03,SNC16} using
lattice points.

Though product weights are easy to work with, they may not be the
appropriate model to describe the dimension structure of the target
function, as we now explain. In accordance with the concepts of effective
dimension \cite{CMO97} and multivariate decomposition \cite{KSWW10a},
every function in $d$ dimensions can be written (in more than one way) as
a sum of $2^d$ terms, $f = \sum_{\setu\subseteq\{1:d\}} f_\setu$, where
each term $f_\setu$ depends only on a subset $\setu$ of the $d$ variables,
namely, $x_j$ for $j\in\setu$. By using an appropriate orthogonal
decomposition for the function space, the terms $f_\setu$ are mutually
orthogonal. We can moderate the importance of each term $f_\setu$ by using
a weight parameter $\gamma_\setu$ for each subset $\setu$ in the function
space definition. These weights $\gamma_\setu$ are called \emph{general
weights} \cite{DSWW06,DKS13}. Product weights are then the special case of
general weights in which $\gamma_\setu = \prod_{j\in\setu} \gamma_j$, that
is, the weight associated with the group of variables indexed by the
subset $\setu$ is obtained by taking the product of the weights $\gamma_j$
corresponding to the variables $x_j$ with $j\in\setu$. The full generality
of general weights allows more flexibility in modeling the functions, but
comes at an exponential cost in $d$ for the CBC construction. So
compromises have been made by researchers by imposing further structure on
the weights, including \emph{order dependent weights} where $\gamma_\setu$
depends only on the cardinality of the set $\setu$, and \emph{finite order
weights} where $\gamma_\setu$ is zero for all $\setu$ with cardinality
greater than a prescribed number. More interestingly, recent works on PDEs
with random coefficients \cite{KSS12,DKLNS14,GKNSSS15,KN16,GKNSS18b} have
led to the invention of a new form of weights called \emph{POD weights} --
\emph{product and order dependent weights}, which combines the features of
product weights and order dependent weights, and even to \emph{SPOD
weights} -- \emph{smoothness-driven product and order dependent weights},
which involves an inner structure depending on the smoothness property of
the function space.

A theoretical justification for the CBC construction of lattice rules for
integration under the general weights setting has been known for some time
\cite{DSWW06}, while fast CBC algorithms for POD weights and SPOD weights
have only been developed in recent times, driven by the need in PDE
applications \cite{KSS12,DKLNS14}. The basic model involves an elliptic
PDE with a random coefficient \cite{CDS10,KSS12,KN16} which is
parameterized by a sequence of stochastic variables (our integration
variables) and the goal is to compute the expected value (an integral with
respect to the large number or even an infinite number of stochastic
variables) of a linear functional $G(\cdot)$ of the PDE solution $u$ with
respect to the spatial variables (with spatial dimension $1$, $2$, or
$3$). To be able to apply the known integration error bounds for lattice
rules, a key step in the analysis is to estimate the norm of the integrand
$f = G(u)$. This requires us to ``differentiate the PDE''
\cite{CDS10,KN16}, to obtain the regularity of $G(u)$ with respect to the
stochastic variables. Estimates of the norm and our desire for
dimension-independent error bounds together lead to the choice of POD
weights or SPOD weights for the function space, and in turn create the
need to construct lattice rules appropriate to these weights.

\subsection*{New contribution}

Motivated by the strong desire to obtain higher order moments or other
statistics of the quantities of interest rather than just the expected
value, we seek in future work to apply lattice algorithms directly to the
PDE solution $u$ at all spatial points as a function of the stochastic
variables. However, all presently available theory on lattice algorithms
for approximation has been for the unweighted setting or just with product
weights. So to proceed we must
\begin{itemize}
\item provide a theoretical justification in the periodic setting for
    the CBC construction of lattice algorithms for approximation
   with general weights; and
\item develop the fast CBC algorithms for the construction of lattice
    algorithms with special structure of weights, especially POD
    weights and SPOD weights.
\end{itemize}
This paper will address the first point, while a companion paper
\cite{CKNS-part2} will address the second point. Both papers involve novel
elements and significant new results that cannot be obtained by trivial
generalizations of existing results.

Specifically, in the periodic Hilbert space setting where the squared
Fourier coefficients decay at the rate of $\alpha>1$, the optimal
convergence rate for integration is $\calO(n^{-\alpha/2+\delta})$,
$\delta>0$, as mentioned earlier, see \cite{SW01}. The optimal algorithm
for $L_2$ approximation in this setting based on the class of
\emph{arbitrary linear information} (implying that all Fourier
coefficients can be obtained exactly) can achieve the same convergence
rate $\calO(n^{-\alpha/2+\delta})$, $\delta>0$, see \cite{NSW04}. However,
if we restrict to the class of \emph{standard information} where only
function values are available, then it has been an open problem whether
the same rate can be achieved with no dependence of the error bound on the
dimension $d$. A general (non-constructive) result in \cite{KWW09a} yields
the convergence rate $\calO(n^{-(\alpha/2)[1/(1+1/\alpha)]+\delta})$,
$\delta>0$, which is nearly optimal for large $\alpha$ but loses a factor
of nearly $1/2$ in the rate when $\alpha$ is small. A very recent
manuscript \cite{KUnew} appears to have solved this open problem.

For algorithms that use function values at lattice points, it is proved
that the best possible convergence rate is $\calO(n^{-\alpha/4+\delta})$,
$\delta>0$; see \cite{BKUV17} for a lower bound which proved the
unavoidable gap in the convergence rates between integration and
approximation. We prove in this paper that a generating vector for a
lattice algorithm and general weights can be obtained by a CBC
construction to achieve this best possible error bound
\[
  \calO(n^{-\alpha/4+\delta}), \quad\delta>0,
\]
with the implied constant independent of $d$, provided that the general
weights satisfy
the condition
$\sum_{\setu\subset\bbN,\,|\setu|<\infty}
\max(|\setu|,1)\, \gamma_\setu^\lambda\,
[2\zeta(\alpha\lambda)]^{|\setu|}<\infty$, where $\lambda =
1/(\alpha-4\delta)$. Here the summation is over all finite subsets of
positive integers $\bbN := \{1,2,\ldots\}$, $|\setu|$~denotes the
cardinality of the set $\setu$, and $\zeta(x) := \sum_{h=1}^\infty h^{-x}$
denotes the Riemann zeta function.

At this point the result can only be said to be semi-constructive, in that
a CBC construction with fully general weights has a prohibitively high
computational cost. In our companion paper \cite{CKNS-part2} we develop
fast CBC algorithms for weights with special structure, including
so-called POD weights and SPOD weights.

Though the best possible convergence rate for algorithms based on lattice
points cannot match the rate of a general optimal algorithm in this
setting (i.e., $\calO(n^{-\alpha/4+\delta})$ versus
$\calO(n^{-\alpha/2+\delta})$, $\delta>0$, see above), lattice-based
algorithms have a number of advantages including simplicity and efficiency
in applications, making them still attractive and competitive.

\section{Problem formulation and review of known results} \label{sec:form}

\subsection{Lattice rules and lattice algorithms}

We consider one-periodic real-valued $L_2$ functions defined on $[0,1]^d$
with absolutely convergent Fourier series
\begin{align*}
  f(\bsx) \,=\, \sum_{\bsh\in\bbZ^d} \hat{f}_\bsh\,e^{2\pi\ri\bsh\cdot\bsx},
  \quad\mbox{with}\quad
  \hat{f}_\bsh \,:=\, \int_{[0,1]^d} f(\bsx)\,e^{-2\pi\ri\bsh\cdot\bsx}\,\rd\bsx,
\end{align*}
where $\hat{f}_\bsh$ are the Fourier coefficients and $\bsh\cdot\bsx =
h_1x_1 + \cdots + h_dx_d$ denotes the usual dot product.

A (rank-$1$) \emph{lattice rule} \cite{SJ94} with $n$ points and
generating vector $\bsz\in \{1,\ldots,n-1\}^d$ approximates the integral
of $f$ by
\begin{align*}
  I(f) \,:=\, \int_{[0,1]^d} f(\bsx)\,\rd\bsx
  \quad\approx\quad
  Q(f) \,:=\, \frac{1}{n} \sum_{k=1}^n f\Big(\Big\{\frac{k\bsz}{n}\Big\}\Big),
\end{align*}
where the braces around a vector indicate that we take the fractional part
of each component in the vector. Using the \emph{character property}
\begin{align*}
  \frac{1}{n} \sum_{k=1}^n e^{2\pi\ri k\bsh\cdot\bsz/n} \,=\,
  \begin{cases}
  1 & \mbox{if } \bsh\cdot\bsz\equiv_n 0, \\
  0 & \mbox{if } \bsh\cdot\bsz\not\equiv_n 0,
  \end{cases}
\end{align*}
it is easy to show that the integration error is
\begin{align} \label{eq:int-err}
  Q(f) - I(f)
  \,=\, \sum_{\satop{\bsh\in\bbZ^d\setminus\{\bszero\}}{\bsh\cdot\bsz\equiv_n 0}} \hat{f}_\bsh,
\end{align}
where $\equiv_n$ denotes congruence modulo $n$.

A \emph{lattice algorithm} for multivariate approximation \cite{KSW06}
with $n$ points and generating vector $\bsz\in \{1,\ldots,n-1\}^d$,
together with an index set $\calA_d\subset \bbZ^d$, approximates the
function $f$ by first truncating the Fourier series to the finite index
set and then approximating the remaining Fourier coefficients by the
lattice cubature points:
\begin{align} \label{eq:Af}
  A(f)(\bsx) \,:=\, \sum_{\bsh\in\calA_d} \hat{f}_\bsh^a \,e^{2\pi\ri\bsh\cdot\bsx},
  \quad\mbox{with}\quad
  \hat{f}_\bsh^a \,:=\, \frac{1}{n} \sum_{k=1}^n f\Big(\Big\{\frac{k\bsz}{n}\Big\}\Big)\,e^{-2\pi\ri k\bsh\cdot\bsz/n}.
\end{align}
The approximation error is
\begin{align*}
  (f - A(f))(\bsx) \,=\, \sum_{\bsh\not\in\calA_d} \hat{f}_\bsh\,e^{2\pi\ri\bsh\cdot\bsx}
  + \sum_{\bsh\in\calA_d} (\hat{f}_\bsh - \hat{f}_\bsh^a)\,e^{2\pi\ri\bsh\cdot\bsx}.
\end{align*}
When measured in the $L_2$ norm over $[0,1]^d$ this leads to
\begin{align} \label{eq:L2-err}
  \|f - A(f)\|_{L_2}^2 \,=\, \sum_{\bsh\not\in\calA_d} \big|\hat{f}_\bsh\big|^2
  + \sum_{\bsh\in\calA_d} \big|\hat{f}_\bsh - \hat{f}_\bsh^a\big|^2.
\end{align}

\subsection{Function space setting with general weights}

For $\alpha>1$ and nonnegative weight parameters
$\bsgamma=\{\gamma_\setu\}$, we consider the Hilbert space $H_d$ of
one-periodic real-valued $L_2$ functions defined on $[0,1]^d$ with
absolutely convergent Fourier series, with norm defined by
\begin{align} \label{eq:r-def}
  \|f\|_d^2 \,:=\, \sum_{\bsh\in\bbZ^d} \big|\hat{f}_\bsh\big|^2\, r(\bsh),
  \quad\mbox{with}\quad
  r(\bsh) \,:=\, \frac{1}{\gamma_{\supp(\bsh)}}\,\prod_{j\in \supp(\bsh)} |h_j|^\alpha,
\end{align}
where $\supp(\bsh) := \{1\le j\le d : h_j \ne 0\}$. The parameter $\alpha$
characterizes the rate of decay of the squared Fourier coefficients, so it
is a smoothness parameter. We fix the scaling of the weights by setting
$\gamma_\emptyset := 1$, so that the norm of a constant function in $H_d$
matches its $L_2$ norm.

Some authors refer to this as the \emph{weighted Korobov space}, see
\cite{SW01} for product weights and \cite{DSWW06} for general weights,
while others call this a weighted variant of the \emph{periodic Sobolev
space with dominating mixed smoothness} \cite{BKUV17}.

When $\alpha\geq 2$ is an even integer, it can be shown that
\begin{align*}
 \|f\|_d^2 \,=\,
 \sum_{\setu\subseteq\{1:d\}}\frac{1}{(2\pi)^{\alpha|\setu|}}\frac{1}{\gamma_{\setu}}
 \int_{[0,1]^{|\setu|}}\!\!
 \bigg(\int_{[0,1]^{d-|\setu|}}\bigg(\prod_{j\in\setu}\frac{\partial}{\partial x_j}\bigg)^{\alpha/2}f(\bsx)
 \,\rd\bsx_{\{1:d\}\setminus\setu}\bigg)^2
 \rd\bsx_{\setu}.
\end{align*}
So $f$ has mixed partial derivatives of order $\alpha/2$. Here $\bsx_\setu
= (x_j)_{j\in\setu}$.

\subsection{Integration}

For the integration problem in the worst case setting, the \emph{worst
case integration error} satisfies
\begin{align*}
  e^{\rm wor\mbox{-}int}_{n,d}(\bsz) \,:=\, \sup_{f\in H_d,\,\|f\|_d\le 1} |I(f) - Q(f)|
  \,=\, \bigg(\sum_{\satop{\bsh\in\bbZ^d\setminus\{\bszero\}}{\bsh\cdot\bsz\equiv_n 0}} \frac{1}{r(\bsh)}\bigg)^{1/2}.
\end{align*}
The initial integration error is $e^{\rm wor\mbox{-}int}_{0,d} \,:=\,
\sup_{f\in H_d,\,\|f\|_d\le 1} |I(f)| = 1$. It is proved in \cite{DSWW06}
that for general weights $\gamma_\setu$, if $n$ is prime, a generating
vector $\bsz$ can be obtained by a CBC construction to achieve the
integration error bound
\begin{align*}
  |I(f) - Q(f)|
  \,\le\, \bigg(\frac{1}{n-1} \sum_{\setu\subset\bbN,\,|\setu|<\infty} \gamma_\setu^\lambda\,
  [2\zeta(\alpha\lambda)]^{|\setu|} \bigg)^{1/(2\lambda)}
  \|f\|_d
  \quad\mbox{for all } \lambda\in (\tfrac{1}{\alpha},1].
\end{align*}
The result generalizes to non-prime $n$, with $n-1$ replaced by the Euler
totient function $\varphi_{\rm tot}(n) := \{1\le z\le n: \gcd(z,n)=1\}$.
Fast CBC algorithms for integration with product weights, order dependent
weights, POD weights and SPOD weights have been developed in
\cite{NC06a,CKN06,KSS11,KKS}.

\subsection{Approximation}

For the approximation problem we can follow \cite{KSW06,KSW08} to define
the index set $\calA_d$ with some parameter $M>0$ by
\begin{align} \label{eq:AdM}
  \calA_d(M) \,:=\, \big\{\bsh\in\bbZ^d : r(\bsh) \le M \big\},
\end{align}
with the difference being that here we have general weights determining
the values of $r(\bsh)$ in \eqref{eq:r-def} while \cite{KSW06,KSW08}
considered product weights. We can then bound the first sum in the $L_2$
approximation error \eqref{eq:L2-err} by
\begin{align*}
  \sum_{\bsh\not\in\calA_d(M)} \big|\hat{f}_\bsh\big|^2
  \,=\, \sum_{\bsh\not\in\calA_d(M)} \big|\hat{f}_\bsh\big|^2 \,r(\bsh)\,\frac{1}{r(\bsh)}
  \,\le\, \|f\|_d^2 \, \frac{1}{M},
\end{align*}
since $r(\bsh)> M$ for $\bsh\notin\calA_d(M)$. The second sum in
\eqref{eq:L2-err} contains the integration error of the function
$g_\bsh(\bsx) := f(\bsx) \,e^{-2\pi\ri\bsh\cdot\bsx}$ so from
\eqref{eq:int-err} we obtain
\begin{align*}
  \big|\hat{f}_\bsh - \hat{f}_\bsh^a\big|^2
  &\,=\, | I(g_\bsh) - Q(g_\bsh)|^2
  \,=\, \bigg| \sum_{\satop{\bsell\in\bbZ^d\setminus\{\bszero\}}{\bsell\cdot\bsz\equiv_n 0}}
  (\widehat{g_\bsh})_\bsell \bigg|^2
  \,=\, \bigg| \sum_{\satop{\bsell\in\bbZ^d\setminus\{\bszero\}}{\bsell\cdot\bsz\equiv_n 0}}
  \hat{f}_{\bsh+\bsell} \bigg|^2 \\
  &\,\le\, \bigg( \sum_{\satop{\bsell\in\bbZ^d\setminus\{\bszero\}}{\bsell\cdot\bsz\equiv_n 0}}
  \big|\hat{f}_{\bsh+\bsell} \big|^2 \, r(\bsh+\bsell)\bigg)\,
  \bigg(\sum_{\satop{\bsell\in\bbZ^d\setminus\{\bszero\}}{\bsell\cdot\bsz\equiv_n 0}} \frac{1}{r(\bsh+\bsell)}\bigg)
  \\
  &\,\le\, \|f\|_d^2 \;
  \sum_{\satop{\bsell\in\bbZ^d\setminus\{\bszero\}}{\bsell\cdot\bsz\equiv_n 0}} \frac{1}{r(\bsh+\bsell)},
\end{align*}
leading to
\begin{align*}
  \sum_{\bsh\in\calA_d(M)} \big|\hat{f}_\bsh - \hat{f}_\bsh^a\big|^2
  &\,\le\, \|f\|_d^2 \; E_d(\bsz),
  \quad\mbox{with}\quad
  E_d(\bsz)\,:=\, \sum_{\bsh\in\calA_d(M)} \sum_{\satop{\bsell\in\bbZ^d\setminus\{\bszero\}}{\bsell\cdot\bsz\equiv_n 0}}
  \frac{1}{r(\bsh+\bsell)}.
\end{align*}
Combining these bounds yields the \emph{worst case $L_2$ approximation}
error bound
\begin{align} \label{eq:bal}
  e^{\rm wor\mbox{-}app}_{n,d,M}(\bsz) \,:=\, \sup_{f\in H_d,\,\|f\|_d\le 1} \|f - A(f)\|_{L_2}
  \,\le\, \bigg(\frac{1}{M} + E_d(\bsz)\bigg)^{1/2}.
\end{align}
More precisely, it was proved in \cite{KSW06} that $e^{\rm
wor\mbox{-}app}_{n,d,M}(\bsz) \le 1/M + \varrho(\calT_\bsz)$, where
$\varrho(\calT_\bsz)$ denotes the spectral radius of some matrix
$\calT_\bsz$ depending on the generating vector~$\bsz$. The $1/M$ term
arose from the truncation to the finite index set, while the spectral
radius arose from the cubature approximations of the remaining
coefficients. Though the elements of the matrix $\calT_\bsz$ were known
explicitly, there was no simple expression for the spectral radius and
therefore it was upper bounded by its trace, leading to the quantity
$E_d(\bsz)$. The initial approximation error is given by $e^{\rm
wor\mbox{-}app}_{0,d} := \sup_{f\in H_d,\,\|f\|_d\le 1} \|f\|_{L_2} =
\max_{\setu\subseteq\{1:d\}} \gamma_\setu^{1/2}$.

It is worth noting that $n$ needs to be large enough in relation to $M$.
For example, we must have $(n-1)^\alpha / \gamma _{\{1\}} >M$, since
otherwise $\bsh^* := (n-1,0,\ldots,0)$ belongs to $\calA_d(M)$ because
$r(\bsh^*) = (n-1)^\alpha / \gamma_{\{1\}}$, as a result of which the sum
over $\bsh$ and $\bsell$ in $E_d(\bsz)$ contains a pair
$(\bsh^*,\bsell^*)$ with $\bsell^* := (-n,0,\cdots,0)$ which contributes
the value of $\gamma_{\{1\}}$ in $E_d(\bsz)$, leading to the sum not
converging to zero as $n \to \infty$. This is ensured below by the
condition $n \ge \kappa M^{1/\alpha}$.

For product weights it is proved in \cite{KSW06,KSW08} that for $n$ prime
a generating vector $\bsz$ can be obtained by a CBC construction to
achieve
\begin{align*}
  E_d(\bsz) \,\le\, \bigg(\frac{1}{\mu^\lambda} \frac{|\calA_d(M)|}{n-1}
  \prod_{j=1}^d \Big( \big(1 + 2(1+\mu^\lambda)\big)\,\zeta(\alpha\lambda)\gamma_j^\lambda\Big) \bigg)^{1/\lambda}
  \quad\mbox{for all } \lambda\in (\tfrac{1}{\alpha},1],
\end{align*}
and for all $\mu\in (0,(1-1/\kappa)^\alpha]$ where $\kappa>1$ is such that
$n\ge \kappa M^{1/\alpha}$. Combining this with the bounds on the
cardinality of the index set \cite{KSW06,KSW08}
\begin{align*}
 (\gamma_1 M)^{1/\alpha} \,\le\,
 |\calA_d(M)| \,\le\, M^q \prod_{j=1}^d \big(1 + 2\zeta(\alpha q)\gamma_j^q\big)
  \quad\mbox{for all } q > \tfrac{1}{\alpha},
\end{align*}
we balance the two terms in \eqref{eq:bal} by taking
\begin{align*}
  \frac{1}{M} \,=\, \frac{M^{q/\lambda}}{n^{1/\lambda}}
  \quad\mbox{and}\quad q = \lambda\in (\tfrac{1}{\alpha},1],
\end{align*}
to obtain the convergence rate $e^{\rm wor\mbox{-}app}_{n,d,M}(\bsz) =
\calO(n^{-1/(4\lambda)}) = \calO(n^{-\alpha/4+\delta})$, $\delta>0$, with
$\lambda = 1/(\alpha-4\delta)$, where the implied constant is independent
of $d$ provided that $\sum_{j\ge 1} \gamma_j^\lambda < \infty$. The cost
of the fast CBC algorithm based on $E_d(\bsz)$ with product weights  is
$\calO(|\calA_d(M)|\,d\,n\log(n))$ operations.

Alternatively, since $r(\bsh)\le M$ for $\bsh\in\calA_d(M)$, we can bound
$E_d(\bsz)$ by
\begin{align} \label{eq:Sd}
  E_d(\bsz)
  &\,\le\, \sum_{\bsh\in\calA_d(M)} \frac{M}{r(\bsh)}
  \sum_{\satop{\bsell\in\bbZ^d\setminus\{\bszero\}}{\bsell\cdot\bsz\equiv_n 0}} \frac{1}{r(\bsh+\bsell)} \nonumber\\
  &\,\le\, M\,S_d(\bsz),
  \quad\mbox{with}\quad
  S_d(\bsz)\,:=\, \sum_{\bsh\in\bbZ^d} \frac{1}{r(\bsh)}
  \sum_{\satop{\bsell\in\bbZ^d\setminus\{\bszero\}}{\bsell\cdot\bsz\equiv_n 0}} \frac{1}{r(\bsh+\bsell)}.
\end{align}
A variant of the quantity $S_d(\bsz)$ first appeared in the context of a
Lattice-Nystr\"om method for Fredholm integral equations of the second
kind \cite{DKKS07}. (Note, however, that in \cite{DKKS07} the quantity was
defined as the square root of the double sum.) The advantage of working
with $S_d(\bsz)$ instead of $E_d(\bsz)$ is that \emph{there is no
dependence on the index set $\calA_d(M)$, thus the error analysis is
simpler and the construction cost is lower.}

For product weights it is proved in \cite{DKKS07} that for $n$ prime a
generating vector $\bsz$ can be obtained by a CBC construction to achieve
\begin{align} \label{eq:Sd-cbc-prod}
  S_d(\bsz) \,\le\, \bigg(\frac{1}{\mu^{2\lambda}} \frac{1}{n}
  \prod_{j=1}^d \Big( \big(1 + 2(1+\mu^\lambda)^{1/2}\big)\,\zeta(\alpha\lambda)\gamma_j^\lambda\Big)^2 \bigg)^{1/\lambda}
  \quad\mbox{for all } \lambda\in (\tfrac{1}{\alpha},1],
\end{align}
and for all $\mu\in (0,2^{-3\alpha}]$. Now we balance the two terms in
\eqref{eq:bal} by taking
\begin{align} \label{eq:chooseM}
  \frac{1}{M} \,=\, \frac{M}{n^{1/\lambda}}
  \quad\mbox{and}\quad \lambda\in (\tfrac{1}{\alpha},1],
\end{align}
to obtain again the convergence rate $e^{\rm wor\mbox{-}app}_{n,d,M}(\bsz)
= \calO(n^{-1/(4\lambda)}) = \calO(n^{-\alpha/4+\delta})$, $\delta>0$,
with $\lambda = 1/(\alpha-4\delta)$, where the implied constant is
independent of $d$ provided that $\sum_{j\ge 1} \gamma_j^\lambda <
\infty$. The cost of the fast CBC algorithm based on $S_d(\bsz)$ with
product weights is only $\calO(d\,n\log(n))$ operations.

The goal of this paper is to obtain an analogous error bound for
$S_d(\bsz)$ with general weights. This is our Theorem~\ref{thm:main}
below.

Note that one can of course apply lattices that are designed for
integration with general weights directly for function approximation, but
as shown in \cite{KSW06} this will lead to a worse convergence rate
compared to designing lattice algorithms specifically for the purpose of
approximation.

\subsection{Other related results}

In this paper we consider $L_2$ approximation in the worst case setting.
Instead of measuring the error in the $L_2$ norm, one can also consider
other $L_p$ norms, including the $L_\infty$ norm \cite{BKUV17}. Instead of
the worst case setting, one can also consider the average case setting
where the function space is equipped with a Gaussian probability measure
with a prescribed mean and covariance function \cite{KSW08}. As already
mentioned earlier, instead of attempting to reduce the error criterion
directly, one can look for \emph{reconstruction lattices} which exactly
reproduce functions whose Fourier series are solely supported on a finite
index set \cite{Kam13,KPV15,PV16}. Instead of the periodic setting, one
can also consider the related \emph{cosine space} or \emph{Chebyshev
space} of nonperiodic functions \cite{SNC16,CKNS16,PV15,KMNN}. One can
also consider lattice algorithms in the context of discrete least square
approximation \cite{KMNN}.

Also related are \emph{spline algorithms} or \emph{kernel methods}
\cite{Wah90,ZLH06,ZKH09} and \emph{collocation} \cite{LH03,SNC16} based on
lattice points. In a reproducing kernel Hilbert space with a
``shift-invariant'' kernel (as we have in the periodic setting here), the
structure of the lattice points allows the required linear system to be
solved in $\calO(n\,\log(n))$ operations. Since splines have the smallest
worst case $L_2$ approximation error among all algorithms that make use of
the same sample points (see for example \cite{ZKH09}), the lattice
generating vectors from this paper can be used in a spline algorithm and
the worst case error bound from this paper will carry over as an immediate
upper bound with no further multiplying constant. The advantage of a
spline over the lattice algorithm \eqref{eq:Af} is that there is no
presence of the index set $\calA_d$ so is extremely efficient in practice.

The best possible rate of convergence for lattice algorithms for
approximation is proved recently in \cite{BKUV17} to be only \emph{half}
of the optimal rate of convergence for lattice rules for integration
(i.e., $\calO(n^{-\alpha/4+\delta})$ versus $\calO(n^{-\alpha/2+\delta})$,
$\delta>0$). This is a negative point for lattice algorithms, since there
are other approximation algorithms such as Smolyak algorithms or sparse
grids which do not suffer from this loss of convergence rate. However, as
discussed in \cite{BKUV17}, lattice algorithms have their advantages in
terms of simplicity in construction and point generation, and stability
and efficiency in application, making them still attractive and
competitive despite the reduced convergence rate.

\section{New results on approximation with general weights} \label{sec:cbc}

\subsection{Size of the index set}

Although the index set $\calA_d(M)$ does not appear in the expression for
$S_d(\bsz)$ in \eqref{eq:Sd}, it does impact the lattice algorithm $A(f)$
defined in \eqref{eq:Af}. Here we provide a bound on its cardinality.

\begin{lemma} \label{lem:AdM}
For all $d\ge 1$, $M>0$ and $q> 1/\alpha$ we have
\begin{align*}
  |\calA_d(M)| \,\le\, M^q \sum_{\setu\subseteq\{1:d\}} [2\zeta(\alpha q)]^{|\setu|}\, \gamma_{\setu}^q.
\end{align*}
\end{lemma}

\begin{proof}
We can write
\begin{align*}
  |\calA_d(M)|
  &\,=\, \sum_{\bsh\in \calA_d(M)} 1
  \,=\, \sum_{\satop{\bsh\in\bbZ^d}{r(\bsh)\le M}} 1
  \,=\, \sum_{\setu\subseteq\{1:d\}}
        \sum_{\satop{\bsh\in\bbZ^d,\, \supp(\bsh)=\setu}{\gamma_{\supp(\bsh)}^{-1} \prod_{j\in \supp(\bsh)} |h_j|^\alpha\le M}} 1 \\
  &\,=\, \sum_{\setu\subseteq\{1:d\}}
        \sum_{\satop{\bsh_\setu\in (\bbZ\setminus\{0\})^{|\setu|}}{\prod_{j\in\setu} |h_j|^\alpha\le \gamma_\setu M}} 1
  \,=\, \sum_{\setu\subseteq\{1:d\}}
        \sum_{\bsh_\setu\in B_\setu(\gamma_\setu M)} 1
  \,=\, \sum_{\setu\subseteq\{1:d\}} |B_\setu(\gamma_\setu M)|,
\end{align*}
where we introduced the auxiliary set (treating $\gamma_\setu$ as part of
the argument $m$)
\begin{align*}
  B_\setu(m) := \bigg\{\bsh_\setu \in(\bbZ\setminus\{0\})^{|\setu|} : \prod_{j\in\setu} |h_j|^\alpha \le m\bigg\},
  \qquad \setu\subseteq\{1:d\}.
\end{align*}
The result holds if we can show that, for all $\setu\subseteq\{1:d\}$,
$m>0$ and $q> 1/\alpha$,
\begin{align*}
  |B_\setu(m)| \,\le\, [2\zeta(\alpha q)]^{|\setu|}\, m^q.
\end{align*}

Since the coordinates are equivalent, it suffices to consider the set
\begin{align*}
  \widetilde{B}_s(m) := \bigg\{\bsh \in(\bbZ\setminus\{0\})^s : \prod_{j=1}^s |h_j|^\alpha \le m\bigg\},
\end{align*}
and show that, for all $s\ge 0$, $m>0$ and $q> 1/\alpha$,
\begin{align} \label{eq:BdM}
  |\widetilde{B}_s(m)| \,\le\, [2\zeta(\alpha q)]^s\, m^q.
\end{align}
We have $|\widetilde{B}_0(m)| = 0$ and $|\widetilde{B}_1(m)| = 2\lfloor
m^{1/\alpha}\rfloor$, so \eqref{eq:BdM} holds trivially for $s=0,1$. For
$s\ge 2$ and assuming that \eqref{eq:BdM} holds with $s$ replaced by
$s-1$, we have
\begin{align*}
  |\widetilde{B}_s(m)|
  &\,=\,  \sum_{\satop{\bsh\in (\bbZ\setminus\{0\})^s}{\prod_{j=1}^s |h_j|^\alpha\le m}} 1
  \,=\,  \sum_{h_s\in\bbZ\setminus\{0\}}
  \sum_{\satop{(h_1,\ldots,h_{s-1})\in (\bbZ\setminus\{0\})^{s-1}}{\prod_{j=1}^{s-1} |h_j|^\alpha\le \frac{m}{|h_s|^\alpha}}} 1 \\
  &\,=\,  \sum_{h_s\in\bbZ\setminus\{0\}} \big|\widetilde{B}_{s-1}\big(\tfrac{m}{|h_s|^\alpha}\big)\big|
  \,\le\,  \sum_{h_s\in\bbZ\setminus\{0\}} [2\zeta(\alpha q)]^{s-1} \big(\tfrac{m}{|h_s|^\alpha}\big)^q
  \,=\,  [2\zeta(\alpha q)]^s\, m^q.
\end{align*}
Hence \eqref{eq:BdM} holds for all $s\ge 0$. This completes the proof.
\end{proof}

\subsection{Dimension-wise decomposition of the error criterion}

With product weights, the error criterion $S_d(\bsz)$ in \eqref{eq:Sd} can
be expressed recursively as
\begin{align*}
  S_d(\bsz)
  \,=\, (1+2\zeta(2\alpha)\gamma_d^2)\, S_{d-1}(z_1,\ldots,z_{d-1}) + \theta_d(\bsz),
\end{align*}
where $\theta_d$ is an expression which captures all the contribution of
the new component~$z_d$. The precise formula for $\theta_d$ in the case of
product weights does not matter here. The main point is that this
recursion provided the inductive step to prove the bound
\eqref{eq:Sd-cbc-prod} for the CBC construction of $\bsz$ based on
minimizing $S_s(z_1,\ldots,z_s)$ for each $s = 1,2,\ldots, d$. Moreover,
this means that the result for product weights is extensible in $d$.

The situation with general weights is quite different: the recursion turns
out to be rather complicated because ``future'' weights get tangled up! To
enable us to describe this complication, we introduce a temporary
notation, $S_d(\bsz) =
S_d(\bsz;\{\gamma_\setu\}_{\setu\subseteq\{1:d\}})$, to show its explicit
dependence on the weights $\{\gamma_\setu\}$. We show in the proof of the
following lemma that, with respect to any input sequence $\{\beta_\setu\}$
(``replaceable'' in every function call), we have
\begin{align} \label{eq:recur}
  S_d \big(\bsz;\{\beta_\setu\}_{\setu\subseteq\{1:d\}})
  &\,=\, S_{d-1} \big(z_1,\ldots,z_{d-1};\{\beta_{\setu}\}_{\setu\subseteq\{1:d-1\}}\big) \nonumber\\
  &\quad\; + 2\zeta(2\alpha)\,S_{d-1} \big(z_1,\ldots,z_{d-1};\{\beta_{\setu\cup\{d\}}\}_{\setu\subseteq\{1:d-1\}}\big) \nonumber\\
  &\quad\; + \theta_d \big(\bsz;\{\beta_{\setu}\}_{\setu\subseteq\{1:d\}}\big),
\end{align}
with $\theta_d$ defined as in \eqref{eq:theta} below. That is, the error
criterion in $d$ dimensions with input sequence $\{\beta_\setu\}$ depends
on the error criterion in $d-1$ dimensions with input sequence
$\{\beta_\setu\}$, as well as on the error criterion in $d-1$ dimensions
in which each parameter $\beta_\setu$ in the input sequence is
``replaced'' by a corresponding ``future'' parameter
$\beta_{\setu\cup\{d\}}$.

This dependence on ``future'' weights means that a CBC construction which
minimizes $S_s(z_1,\ldots,z_s)$ one dimension at a time cannot work here,
because there is no way to establish a valid induction argument! To
overcome this difficulty, we need to fix \textit{a priori} a value of $d$
to be the target final dimension, and then use the recursion
\eqref{eq:recur} to decompose
$S_d(\bsz;\{\gamma_\setu\}_{\setu\subseteq\{1:d\}})$ all the way down to
the first dimension, to yield a \emph{dimension-wise decomposition} of the
error criterion as shown below. A similar strategy has previously been
used in \cite{NSW17,ELN18}.

\begin{lemma}
Let $d\ge 1$ be fixed and a sequence of weights
$\{\gamma_\setu\}_{\setu\subseteq\{1:d\}}$ be given. We can write
\begin{align} \label{eq:Sd-decomp}
  S_d(\bsz)
  \,=\, \sum_{s=1}^d T_{d,s} \big(z_1,\ldots,z_s\big),
\end{align}
where, for each $s=1,2,\ldots,d$,
\begin{align} \label{eq:Tds}
  T_{d,s} \big(z_1,\ldots,z_s\big)
  \,:=\, \sum_{\setw\subseteq\{s+1:d\}}
  [2\zeta(2\alpha)]^{|\setw|}\, \theta_s \big(z_1,\ldots,z_s;\{\gamma_{\setu\cup\setw}\}_{\setu\subseteq\{1:s\}}\big),
\end{align}
\begin{align} \label{eq:theta}
  \theta_s \big(z_1,\ldots,z_s;\{\beta_\setu\}_{\setu\subseteq\{1:s\}}\big)
  \,:=\, \sum_{\bsh\in\bbZ^s} \sum_{\satop{\bsell\in\bbZ^s,\;\ell_s\ne 0}{\bsell\cdot (z_1,\ldots,z_s)\equiv_n 0}}
  \frac{\beta_{\supp(\bsh)}}{r'(\bsh)}
  \frac{\beta_{\supp(\bsh+\bsell)}}{r'(\bsh+\bsell)},
\end{align}
with $r'(\bsh) := \prod_{j\in\supp(\bsh)} |h_j|^\alpha$.
\end{lemma}

\begin{proof}
We remark that the quantity $r'(\bsh)$ is essentially $r(\bsh)$ without
the weight parameter. We generalize the definition in \eqref{eq:Sd}: for
each $s=1,2,\ldots$ and any input sequence
$\{\beta_\setu\}_{\setu\subseteq\{1:s\}}$, we define
\begin{align*}
  S_s \big(z_1,\ldots,z_s; \{\beta_\setu\}_{\setu\subseteq\{1:s\}}\big)
  \,:=\, \sum_{\bsh\in\bbZ^s} \sum_{\satop{\bsell\in\bbZ^s\setminus\{\bszero\}}{\bsell\cdot (z_1,\ldots,z_s)\equiv_n 0}}
  \frac{\beta_{\supp(\bsh)}}{r'(\bsh)}
  \frac{\beta_{\supp(\bsh+\bsell)}}{r'(\bsh+\bsell)},
\end{align*}
so that $S_d \big(z_1,\ldots,z_d;
\{\gamma_\setu\}_{\setu\subseteq\{1:d\}}\big)$ agrees with \eqref{eq:Sd}.
We define additionally $S_0 := 0$. By separating the cases (i) $h_s=\ell_s
= 0$, (ii) $h_s\ne 0$ and $\ell_s=0$, (iii) $\ell_s\ne 0$, we obtain
\begin{align*}
  &S_s \big(z_1,\ldots,z_s; \{\beta_\setu\}_{\setu\subseteq\{1:s\}}\big) \\
  &\qquad\,=\, \sum_{\bsh\in\bbZ^{s-1}} \sum_{\satop{\bsell\in\bbZ^{s-1}\setminus\{\bszero\}}{\bsell\cdot (z_1,\ldots,z_{s-1})\equiv_n 0}}
  \frac{\beta_{\supp(\bsh)}}{r'(\bsh)}
  \frac{\beta_{\supp(\bsh+\bsell)}}{r'(\bsh+\bsell)} \\
  &\qquad\qquad + \sum_{h_s\in\bbZ\setminus\{0\}} \frac{1}{|h_s|^{2\alpha}}
  \sum_{\bsh\in\bbZ^{s-1}} \sum_{\satop{\bsell\in\bbZ^{s-1}\setminus\{\bszero\}}{\bsell\cdot (z_1,\ldots,z_{s-1})\equiv_n 0}}
  \frac{\beta_{\supp(\bsh)\cup\{s\}}}{r'(\bsh)}
  \frac{\beta_{\supp(\bsh+\bsell)\cup\{s\}}}{r'(\bsh+\bsell)} \\
  &\qquad\qquad + \theta_s \big(z_1,\ldots,z_s;\{\beta_\setu\}_{\setu\subseteq\{1:s\}}\big) \\
  &\qquad\,=\, S_{s-1}\big(z_1,\ldots,z_{s-1};\{\beta_{\setu}\}_{\setu\subseteq\{1:s-1\}}\big) \\
  &\qquad\qquad + 2\zeta(2\alpha)\,S_{s-1}\big(z_1,\ldots,z_{s-1};\{\beta_{\setu\cup\{s\}}\}_{\setu\subseteq\{1:s-1\}}\big)
  + \theta_s\big(z_1,\ldots,z_s;\{\beta_{\setu}\}_{\setu\subseteq\{1:s\}}\big),
\end{align*}
where $\theta_s$ is as defined in \eqref{eq:theta}. This proves
\eqref{eq:recur} by taking $s=d$.

Abbreviating temporarily the above recursion by
\begin{align*}
 S_s \big(\{\beta_\setu\}_\setu\big)
 \,=\, S_{s-1} \big(\{\beta_{\setu}\}_\setu\big) + c\,S_{s-1} \big(\{\beta_{\setu\cup\{s\}}\}_\setu\big)
 + \theta_s\big(\{\beta_\setu\}_\setu\big),
\end{align*}
with $c := 2\zeta(2\alpha)$, we can write
\begin{align*}
  &S_{d}\big(\{\gamma_\setu\}_\setu\big)
  = S_{d-1}\big(\{\gamma_\setu\}_\setu\big) + c\,S_{d-1}\big(\{\gamma_{\setu\cup\{d\}}\}_\setu\big) + \theta_{d}(\{\gamma_\setu\}_\setu\big) \\
  &\qquad= S_{d-2}\big(\{\gamma_\setu\}_\setu\big) + c\,S_{d-2}\big(\{\gamma_{\setu\cup\{d-1\}}\}_\setu\big) + \theta_{d-1}(\{\gamma_\setu\}_\setu\big) \\
  &\qquad\quad
  + c\,S_{d-2}\big(\{\gamma_{\setu\cup\{d\}}\}_\setu\big) + c^2\,S_{d-2}\big(\{\gamma_{\setu\cup\{d-1,d\}}\}_\setu\big) + c\,\theta_{d-1}(\{\gamma_{\setu\cup\{d\}}\}_\setu\big) \\
  &\qquad\quad + \theta_{d}(\{\gamma_\setu\}_\setu\big) \\
  &\qquad= S_{d-3}\big(\{\gamma_\setu\}_\setu\big) + c\,S_{d-3}\big(\{\gamma_{\setu\cup\{d-2\}}\}_\setu\big) + \theta_{d-2}\big(\{\gamma_\setu\}_\setu\big)\\
  &\qquad\quad + c\,S_{d-3}\big(\{\gamma_{\setu\cup\{d-1\}}\}_\setu\big) + c^2\,S_{d-3}\big(\{\gamma_{\setu\cup\{d-2,d-1\}}\}_\setu\big) + c\,\theta_{d-2}\big(\{\gamma_{\setu\cup\{d-1\}}\}_\setu\big) \\
  &\qquad\quad + \theta_{d-1}(\{\gamma_\setu\}_\setu\big) \\
  &\qquad\quad + c\,S_{d-3}\big(\{\gamma_{\setu\cup\{d\}}\}_\setu\big) + c^2\,S_{d-3}\big(\{\gamma_{\setu\cup\{d-2,d\}}\}_\setu\big) + c\,\theta_{d-2}\big(\{\gamma_{\setu\cup\{d\}}\}_\setu\big) \\
  &\qquad\quad + c^2S_{d-3}\big(\{\gamma_{\setu\cup\{d-1,d\}}\}_\setu\big) + c^3\,S_{d-3}\big(\{\gamma_{\setu\cup\{d-2,d-1,d\}}\}_\setu\big) + c^2\theta_{d-2}\big(\{\gamma_{\setu\cup\{d-1,d\}}\}_\setu\big) \\
  &\qquad\quad + c\,\theta_{d-1}(\{\gamma_{\setu\cup\{d\}}\}_\setu\big) \\
  &\qquad\quad + \theta_{d}(\{\gamma_\setu\}_\setu\big).
\end{align*}
Continuing this way to decompose the terms until we reach $S_0 =0$, we
eventually end up with the expression in the lemma.
\end{proof}

\subsection{Component-by-component construction}

Algorithm~\ref{alg} below outlines a CBC construction for the generating
vector $\bsz$. Lemma~\ref{lem:avg} provides the essential averaging
argument needed in the proof of Theorem~\ref{thm:main}. The main result,
Theorem~\ref{thm:final}, is that we achieve the best possible convergence
rate for lattice algorithms as proven in~\cite{BKUV17}.

\begin{algorithm} \label{alg}
Given $n\ge 2$, a fixed $d\ge 1$, and a sequence of weights
$\{\gamma_\setu\}_{\setu\subseteq\{1:d\}}$, the generating vector $\bsz^*
= (z_1^*,\ldots,z_d^*)$ is constructed as follows: for each $s = 1,
\ldots,d$, with $z_1^*,\ldots,z_{s-1}^*$ fixed, choose $z_s \in
\{1,\ldots,n-1\}$ to minimize $T_{d,s}
\big(z_1^*,\ldots,z_{s-1}^*,z_s\big)$ given by \eqref{eq:Tds}.
\end{algorithm}

\begin{lemma} \label{lem:avg}
Let $n$ be prime. For any $s\ge 1$, any input sequence
$\{\beta_\setu\}_{\setu\subseteq\{1:s\}}$, all values of
$z_1,\ldots,z_{s-1}$, and all $\lambda\in (\frac{1}{\alpha},1]$, we have
\begin{align} \label{eq:theta-bound}
  &\frac{1}{n-1} \sum_{z_s=1}^{n-1}
  \big[\theta_s \big(z_1,\ldots,z_{s-1},z_s;\{\beta_\setu\}_{\setu\subseteq\{1:s\}}\big)\big]^\lambda \nonumber\\
  &\qquad\,\le\, \frac{\tau}{n}
  \bigg(\sum_{s\in\setu\subseteq\{1:s\}} \beta_{\setu}^\lambda\, [2\zeta(\alpha\lambda)]^{|\setu|}\bigg)
  \bigg(
  \sum_{\setu\subseteq\{1:s\}} \beta_\setu^\lambda\, [2\zeta(\alpha\lambda)]^{|\setu|} \bigg),
\end{align}
where $\tau := \max(6,2.5+2^{2\alpha\lambda+1})$.
\end{lemma}

\begin{proof}
We have from the formula \eqref{eq:theta} and Jensen's inequality $\sum_k
a_k\le (\sum_k a_k^\lambda)^{1/\lambda}$ for all $a_k\ge 0$ that
\begin{align} \label{eq:2terms}
  {\rm Avg} \,&:=\, \frac{1}{n-1} \sum_{z_s=1}^{n-1}
  \big[\theta_s \big(z_1,\ldots,z_{s-1},z_s;\{\beta_\setu\}_{\setu\subseteq\{1:s\}}\big)\big]^\lambda \nonumber\\
  &\,\le\, \frac{1}{n-1} \sum_{z_s=1}^{n-1}
  \sum_{\bsh\in\bbZ^s} \sum_{\satop{\bsell\in\bbZ^s,\;\ell_s\ne 0}{(\ell_1,\ldots,\ell_{s-1})\cdot (z_1,\ldots,z_{s-1})\equiv_n -\ell_sz_s}}
  \bigg(\frac{\beta_{\supp(\bsh)}}{r'(\bsh)}
  \frac{\beta_{\supp(\bsh+\bsell)}}{r'(\bsh+\bsell)} \bigg)^\lambda \nonumber\\
  &\,=\, \frac{1}{n-1}
  \sum_{\bsh\in\bbZ^s} \sum_{\satop{\bsell\in\bbZ^s,\;\ell_s\ne 0,\;\ell_s\not\equiv_n 0}{(\ell_1,\ldots,\ell_{s-1})\cdot (z_1,\ldots,z_{s-1})\not\equiv_n 0}}
  \bigg(\frac{\beta_{\supp(\bsh)}}{r'(\bsh)}
  \frac{\beta_{\supp(\bsh+\bsell)}}{r'(\bsh+\bsell)} \bigg)^\lambda \nonumber\\
  &\qquad +
  \sum_{\bsh\in\bbZ^s} \sum_{\satop{\bsell\in\bbZ^s,\;\ell_s\ne 0,\;\ell_s\equiv_n 0}{(\ell_1,\ldots,\ell_{s-1})\cdot (z_1,\ldots,z_{s-1})\equiv_n 0}}
  \bigg(\frac{\beta_{\supp(\bsh)}}{r'(\bsh)}
  \frac{\beta_{\supp(\bsh+\bsell)}}{r'(\bsh+\bsell)} \bigg)^\lambda,
\end{align}
where we separated the terms depending on whether or not $\ell_s$ is a
multiple of $n$. In particular, we used the fact that for $n$ prime and
$\ell_s\not\equiv_n 0$, the product $\ell_sz_s$ covers each number from
$1$ to $n-1$ in some order as $z_s$ runs from $1$ to $n-1$.

Next we obtain an upper bound by dropping the conditions on the dot
product in both terms in \eqref{eq:2terms} (thus dropping all dependence
on $z_1,\ldots,z_{s-1}$), and define
\begin{align*}
  G(h_s,\ell_s) \,:=\, \sum_{\bsh\in\bbZ^{s-1}} \sum_{\bsell\in\bbZ^{s-1}}
  \bigg(\frac{\beta_{\supp(\bsh,h_s)}}{r'(\bsh,h_s)}
  \frac{\beta_{\supp((\bsh,h_s)+(\bsell,\ell_s))}}{r'((\bsh,h_s)+(\bsell,\ell_s))} \bigg)^\lambda,
\end{align*}
so that
\begin{align*}
  {\rm Avg}
  &\,\le\,
  \frac{1}{n-1} \sum_{h_s\in \bbZ} \sum_{\satop{\ell_s\in\bbZ\setminus\{0\}}{\ell_s\not\equiv_n 0}} G(h_s,\ell_s)
  + \sum_{h_s\in \bbZ} \sum_{\satop{\ell_s\in\bbZ\setminus\{0\}}{\ell_s\equiv_n 0}} G(h_s,\ell_s) \\
  &\,\le\,
  \frac{1}{n-1} \underbrace{\sum_{h_s\in \bbZ} \sum_{\ell_s\in\bbZ\setminus\{0\}} G(h_s,\ell_s)}_{=:\,W_1}
  + \underbrace{\sum_{\satop{h_s\in \bbZ}{h_s\equiv_n 0}}
                \sum_{\satop{\ell_s\in\bbZ\setminus\{0\}}{\ell_s\equiv_n 0}}  G(h_s,\ell_s)}_{=:\, W_2}
  + \underbrace{\sum_{\satop{h_s\in \bbZ}{h_s\not\equiv_n 0}}
                \sum_{\satop{\ell_s\in\bbZ\setminus\{0\}}{\ell_s\equiv_n 0}}  G(h_s,\ell_s)}_{=:\, W_3},
\end{align*}
where we further upper bounded by dropping the condition
$\ell_s\not\equiv_n 0$ in the first term and then splitting the remaining
term into the cases $h_s\equiv_n 0$ and $h_s\not\equiv_n 0$.

For $\ell_s\ne 0$, with a relabeling of $\bsq = \bsh+\bsell$, it is
straightforward to show that
\begin{align*}
  G(h_s,\ell_s)
  &\,=\, \bigg(\sum_{\bsh\in \bbZ^{s-1}}  \frac{\beta_{\supp(\bsh,h_s)}^\lambda}{r'(\bsh,h_s)^\lambda} \bigg)
  \bigg(\sum_{\bsq\in\bbZ^{s-1}} \frac{\beta_{\supp(\bsq,h_s+\ell_s)}^\lambda}{r'(\bsq,h_s+\ell_s)^\lambda}\bigg) \\
  &\,=\,
  \begin{cases}
  \displaystyle\frac{1}{|\ell_s|^{\alpha\lambda}} \scrP\,\scrQ
  & \mbox{if } h_s= 0 \mbox{ and } \ell_s \ne 0, \\[4mm]
  \displaystyle\frac{1}{|h_s|^{\alpha\lambda}}\scrQ\,\scrP
  & \mbox{if } h_s\ne 0 \mbox{ and }  \ell_s = -h_s, \\[4mm]
  \displaystyle\frac{1}{|h_s|^{\alpha\lambda}} \frac{1}{|h_s+\ell_s|^{\alpha\lambda}} \scrQ^2
  & \mbox{if } h_s\ne 0 \mbox{ and }  \ell_s \ne -h_s,
  \end{cases}
\end{align*}
with the abbreviations
\begin{align*}
  \scrP \,:=\, \sum_{\setu\subseteq\{1:s-1\}} \beta_\setu^\lambda\, [2\zeta(\alpha\lambda)]^{|\setu|}
  \quad\mbox{and}\quad
  \scrQ \,:=\, \sum_{\setu\subseteq\{1:s-1\}} \beta_{\setu\cup\{s\}}^\lambda\, [2\zeta(\alpha\lambda)]^{|\setu|}.
\end{align*}
Thus
\begin{align*}
  W_1
   \,&:=\, \sum_{h\in \bbZ} \sum_{\ell\in\bbZ\setminus\{0\}} G(h,\ell) \\
  &\,=\, \sum_{\ell\in\bbZ\setminus\{0\}} \frac{\scrP\,\scrQ}{|\ell|^{\alpha\lambda}}
  + \sum_{h\in \bbZ\setminus\{0\}} \frac{\scrQ\,\scrP}{|h|^{\alpha\lambda}}
  + \sum_{h\in \bbZ\setminus\{0\}} \sum_{\ell\in\bbZ\setminus\{0,-h\}}
  \frac{\scrQ^2}{|h|^{\alpha\lambda}\,|h+\ell|^{\alpha\lambda}} \\
  &\,\le\, 2\zeta(\alpha\lambda)\, \scrP\,\scrQ
  + 2\zeta(\alpha\lambda)\,\scrQ\,\scrP
  + \sum_{h\in \bbZ\setminus\{0\}} \sum_{q\in\bbZ\setminus\{0\}}
    \frac{\scrQ^2}{|h|^{\alpha\lambda}\,|q|^{\alpha\lambda}} \\
  &\,=\, 2[2\zeta(\alpha\lambda)]\, \scrP\,\scrQ
  + [2\zeta(\alpha\lambda)]^2\,\scrQ^2,
\end{align*}
and similarly
\begin{align*}
   W_2
   \,&:=\, \sum_{\satop{h\in \bbZ}{h\equiv_n 0}}
                \sum_{\satop{\ell\in\bbZ\setminus\{0\}}{\ell\equiv_n 0}}  G(h,\ell) \\
   &\,=\, \sum_{\ell\in\bbZ\setminus\{0\}} \frac{\scrP\,\scrQ}{|\ell n|^{\alpha\lambda}}
  + \sum_{h\in \bbZ\setminus\{0\}} \frac{\scrQ\,\scrP}{|hn|^{\alpha\lambda}}
  + \sum_{h\in \bbZ\setminus\{0\}} \sum_{\ell\in\bbZ\setminus\{0,-h\}}
  \frac{\scrQ^2}{|hn|^{\alpha\lambda}\,|(h+\ell)n|^{\alpha\lambda}} \\
  &\,\le\, \frac{2\zeta(\alpha\lambda)\, \scrP\,\scrQ}{n^{\alpha\lambda}}
  + \frac{2\zeta(\alpha\lambda)\,\scrQ\,\scrP}{n^{\alpha\lambda}}
  + \sum_{h\in \bbZ\setminus\{0\}} \sum_{q\in\bbZ\setminus\{0\}}
    \frac{\scrQ^2}{|hn|^{\alpha\lambda}\,|qn|^{\alpha\lambda}} \\
  &\,\le\, \frac{2[2\zeta(\alpha\lambda)]\, \scrP\,\scrQ}{n^{\alpha\lambda}}
  + \frac{[2\zeta(\alpha\lambda)]^2\,\scrQ^2}{n^{2\alpha\lambda}}.
\end{align*}
Moreover, we have
\begin{align*}
   W_3
  \,:=\, \sum_{\satop{h\in \bbZ}{h\not\equiv_n 0}}
                \sum_{\satop{\ell\in\bbZ\setminus\{0\}}{\ell\equiv_n 0}} G(h,\ell)
  &\,=\, \sum_{\satop{h\in \bbZ\setminus\{0\}}{h\not\equiv_n 0}} \sum_{\ell\in\bbZ\setminus\{0\}}
  \frac{\scrQ^2}{|h|^{\alpha\lambda}\,|h+\ell n|^{\alpha\lambda}} \\
  &\,=\, \scrQ^2
  \sum_{\satop{h\in \bbZ\setminus\{0\}}{h\not\equiv_n 0}}
  \bigg(\frac{1}{|h|^{\alpha\lambda}}
  \sum_{\ell\in\bbZ} \frac{1}{|h+\ell n|^{\alpha\lambda}}
  - \frac{1}{|h|^{2\alpha\lambda}} \bigg),
\end{align*}
where we separated out the case $\ell = 0$. Writing $h = pn + k$ with $k$
being the remainder modulo $n$, we obtain
\begin{align*}
   W_3
  &\,=\, \scrQ^2
  \sum_{\satop{k=-(n-1)/2}{k\ne 0}}^{(n-1)/2}
  \sum_{p\in\bbZ} \bigg( \frac{1}{|pn+k|^{\alpha\lambda}}
  \sum_{\ell\in\bbZ} \frac{1}{|pn+k+\ell n|^{\alpha\lambda}}
  - \frac{1}{|pn+k|^{2\alpha\lambda}} \bigg)\\
  &\,=\, \scrQ^2
  \sum_{\satop{k=-(n-1)/2}{k\ne 0}}^{(n-1)/2}
  \sum_{p\in\bbZ} \bigg( \frac{1}{|pn+k|^{\alpha\lambda}}
  \sum_{q\in\bbZ} \frac{1}{|qn+k|^{\alpha\lambda}}
  - \frac{1}{|pn+k|^{2\alpha\lambda}} \bigg)\\
  &\,=\, \scrQ^2
  \sum_{\satop{k=-(n-1)/2}{k\ne 0}}^{(n-1)/2}
  \Bigg( \bigg(\sum_{p\in\bbZ} \frac{1}{|pn+k|^{\alpha\lambda}}\bigg)^2
  - \sum_{p\in\bbZ} \frac{1}{|pn+k|^{2\alpha\lambda}} \Bigg)\\
  &\,\le\, \scrQ^2
  \sum_{\satop{k=-(n-1)/2}{k\ne 0}}^{(n-1)/2}
  \Bigg( \bigg(\frac{1}{|k|^{\alpha\lambda}} + \sum_{p\in\bbZ\setminus\{0\}}
  \frac{1}{|pn|^{\alpha\lambda}|1+k/(pn)|^{\alpha\lambda}}\bigg)^2
  - \frac{1}{|k|^{2\alpha\lambda}}\Bigg).
\end{align*}
Now for $|k|\le (n-1)/2$ and $|p|\ge 1$, we have $|1 + k/(pn)|\ge 1/2$,
and so
\begin{align*}
   W_3
  &\,\le\, \scrQ^2
  \sum_{\satop{k=-(n-1)/2}{k\ne 0}}^{(n-1)/2}
  \Bigg( \bigg(\frac{1}{|k|^{\alpha\lambda}} + \sum_{p\in\bbZ\setminus\{0\}} \frac{1}{|pn|^{\alpha\lambda}(1/2)^{\alpha\lambda}}\bigg)^2
  - \frac{1}{|k|^{2\alpha\lambda}}\Bigg)\\
  &\,=\, \scrQ^2
  \sum_{\satop{k=-(n-1)/2}{k\ne 0}}^{(n-1)/2}
  \Bigg( \frac{2}{|k|^{\alpha\lambda}} \frac{2^{\alpha\lambda+1}\zeta(\alpha\lambda)}{n^{\alpha\lambda}}
  + \bigg(\frac{2^{\alpha\lambda+1}\zeta(\alpha\lambda)}{n^{\alpha\lambda}}\bigg)^2 \Bigg)\\
  &\,\le\, \scrQ^2
  \Bigg( 4\zeta(\alpha\lambda) \frac{2^{\alpha\lambda+1}\zeta(\alpha\lambda)}{n^{\alpha\lambda}}
  + (n-1)\bigg(\frac{2^{\alpha\lambda+1}\zeta(\alpha\lambda)}{n^{\alpha\lambda}}\bigg)^2 \Bigg)
  \,\le\, \frac{2^{2\alpha\lambda+1}\,[2\zeta(\alpha\lambda)]^2\,\scrQ^2}{n^{\alpha\lambda}},
\end{align*}
where we used $\alpha\lambda>1$.

Combining the bounds on $W_1$, $W_2$, $W_3$, and using $1/(n-1)\le 2/n$
and $1/n^{\alpha\lambda} \le 1/n$ and $1/n^{2\alpha\lambda} \le 1/(2n)$,
we obtain
\begin{align*}
  {\rm Avg}
  &\,\le\, \frac{2^2[2\zeta(\alpha\lambda)]\, \scrP\,\scrQ + 2[2\zeta(\alpha\lambda)]^2\,\scrQ^2}{n}
  + \frac{2[2\zeta(\alpha\lambda)]\, \scrP\,\scrQ + \frac{1}{2} [2\zeta(\alpha\lambda)]^2\,\scrQ^2}{n} \\
  &\qquad + \frac{2^{2\alpha\lambda+1}\,[2\zeta(\alpha\lambda)]^2\,\scrQ^2}{n} \\
  &\,=\, \frac{6[2\zeta(\alpha\lambda)]\,\scrP\,\scrQ + \tau_0 [2\zeta(\alpha\lambda)]^2\,\scrQ^2}{n},
  \qquad \tau_0 := 2.5+2^{2\alpha\lambda+1}.
\end{align*}
Writing $\tau := \max(6,\tau_0)$, we have
\begin{align*}
  {\rm Avg}
  &\,\le\, \frac{\tau}{n}
  \bigg(2\zeta(\alpha\lambda)\sum_{\setu\subseteq\{1:s-1\}} \beta_{\setu\cup\{s\}}^\lambda\, [2\zeta(\alpha\lambda)]^{|\setu|} \bigg) \\
  &\qquad\qquad \times
  \bigg(\sum_{\setu\subseteq\{1:s-1\}} \beta_{\setu}^\lambda\, [2\zeta(\alpha\lambda)]^{|\setu|}
  + 2\zeta(\alpha\lambda)\sum_{\setu\subseteq\{1:s-1\}} \beta_{\setu\cup\{s\}}^\lambda\, [2\zeta(\alpha\lambda)]^{|\setu|} \bigg) \\
  &\,=\, \frac{\tau}{n}
  \bigg(\sum_{s\in\setu\subseteq\{1:s\}} \beta_{\setu}^\lambda\, [2\zeta(\alpha\lambda)]^{|\setu|} \bigg)
  \bigg(\sum_{\setu\subseteq\{1:s-1\}} \beta_{\setu}^\lambda\, [2\zeta(\alpha\lambda)]^{|\setu|}
  + \sum_{s\in\setu\subseteq\{1:s\}} \beta_{\setu}^\lambda\, [2\zeta(\alpha\lambda)]^{|\setu|} \bigg) \\
  &\,=\, \frac{\tau}{n}
  \bigg(\sum_{s\in\setu\subseteq\{1:s\}} \beta_{\setu}^\lambda\, [2\zeta(\alpha\lambda)]^{|\setu|} \bigg)
  \bigg(\sum_{\setu\subseteq\{1:s\}} \beta_{\setu}^\lambda\, [2\zeta(\alpha\lambda)]^{|\setu|} \bigg).
\end{align*}
This completes the proof.
\end{proof}

\begin{theorem} \label{thm:main}
Let $n$ be prime. For fixed $d\ge 1$ and a given sequence of weights
$\{\gamma_\setu\}_{\setu\subseteq\{1:d\}}$, a generating vector $\bsz$
obtained from the CBC construction following Algorithm~\ref{alg} satisfies
for all $\lambda\in (\tfrac{1}{\alpha},1]$,
\begin{align} \label{eq:final}
  S_d(\bsz) \,\le\,
  \bigg[ \frac{\tau}{n} \bigg(
  \sum_{\emptyset\ne\setu\subseteq\{1:d\}} |\setu|\,\gamma_{\setu}^\lambda\, [2\zeta(\alpha\lambda)]^{|\setu|}\bigg)
  \bigg(\sum_{\setu\subseteq\{1:d\}} \gamma_\setu^\lambda\, [2\zeta(\alpha\lambda)]^{|\setu|}\bigg) \bigg]^{1/\lambda},
\end{align}
where $\tau := \max(6,2.5+2^{2\alpha\lambda+1})$. Furthermore, if the
weights are such that there exists a constant $\xi\ge 1$ $($which may
depend on $\lambda$$)$ such that
\begin{align} \label{eq:decay}
  \gamma_{\setu\cup\setw}^\lambda \le \xi\,\frac{\gamma_\setu^\lambda}{[2\zeta(\alpha\lambda)]^{|\setw|}}
  \quad\mbox{for all}\quad \setu\subseteq\{1:s\},\;  \setw\subseteq\{s+1:d\},\; s\ge 1,\; d\ge 1,
\end{align}
then \eqref{eq:final} holds with $\tau$ replaced by $\tau\,\xi$ and with
the $|\setu|$ factor inside the first sum replaced by $1$.
\end{theorem}

\begin{proof}
Let $\bsz^* = (z_1^*,\ldots,z_d^*)$ denote the generating vector obtained
from Algorithm~\ref{alg}. We have from \eqref{eq:Sd-decomp} that
\begin{align*}
  S_d(\bsz^*)
  \,=\, \sum_{s=1}^d T_{d,s} \big(z_1^*,\ldots,z_s^*\big).
\end{align*}

For each $s=1,\ldots,d$, the component $z_s^*$ is chosen to minimize the
quantity $T_{d,s} \big(z_1^*,\ldots,z_{s-1}^*,z_s\big)$ over all
$z_s\in\{1,\ldots,n-1\}$. Since the minimum must be smaller than or equal
to the average, for all $\lambda\in (\frac{1}{\alpha},1]$ we have
\begin{align*}
  &\big[T_{d,s} \big(z_1^*,\ldots,z_s^*\big)\big]^\lambda
  \,\le\, \frac{1}{n-1} \sum_{z_s=1}^{n-1}
  \big[T_{d,s} \big(z_1^*,\ldots,z_{s-1}^*,z_s\big)\big]^\lambda \\
  &\qquad\,=\, \frac{1}{n-1} \sum_{z_s=1}^{n-1}  \bigg(\sum_{\setw\subseteq\{s+1:d\}} [2\zeta(2\alpha)]^{|\setw|}
  \theta_s \big(z_1^*,\ldots,z_{s-1}^*,z_s;\{\gamma_{\setu\cup\setw}\}_{\setu\subseteq\{1:s\}}\big)\bigg)^\lambda \\
  &\qquad\,\le\, \frac{1}{n-1} \sum_{z_s=1}^{n-1}  \sum_{\setw\subseteq\{s+1:d\}} [2\zeta(2\alpha)]^{\lambda|\setw|}
  \big[\theta_s \big(z_1^*,\ldots,z_{s-1}^*,z_s;\{\gamma_{\setu\cup\setw}\}_{\setu\subseteq\{1:s\}}\big)\big]^\lambda \\
  &\qquad\,=\, \sum_{\setw\subseteq\{s+1:d\}} [2\zeta(2\alpha)]^{\lambda|\setw|}
  \bigg(\frac{1}{n-1} \sum_{z_s=1}^{n-1}
  \big[\theta_s \big(z_1^*,\ldots,z_{s-1}^*,z_s;\{\gamma_{\setu\cup\setw}\}_{\setu\subseteq\{1:s\}}\big)\big]^\lambda \bigg),
\end{align*}
where we used Jensen's inequality.

Now for every $\setw \subseteq \{s+1:d\}$, we apply Lemma~\ref{lem:avg}
with $z_1 = z_1^*$, \ldots, $z_{s-1} = z_{s-1}^*$ and input sequence
$\beta_\setu = \gamma_{\setu\cup\setw}$ for each $\setu\subseteq\{1:s\}$.
In other words, Lemma~\ref{lem:avg} is applied $2^{d-s}$ times, each time
with a different input sequence depending on $\setw$. Using
$[2\zeta(2\alpha)]^{\lambda} \le [2\zeta(\alpha\lambda)]^2$ and
\eqref{eq:theta-bound}, we obtain
\begin{align*}
  &\big[T_{d,s} \big(z_1^*,\ldots,z_s^*\big)\big]^\lambda \\
  &\qquad\,\le\, \sum_{\setw\subseteq\{s+1:d\}} [2\zeta(\alpha\lambda)]^{2|\setw|}
  \frac{\tau}{n}
  \bigg(\sum_{s\in\setu\subseteq\{1:s\}} \gamma_{\setu\cup\setw}^\lambda\, [2\zeta(\alpha\lambda)]^{|\setu|}\bigg)
  \bigg(\sum_{\setu\subseteq\{1:s\}} \gamma_{\setu\cup\setw}^\lambda\, [2\zeta(\alpha\lambda)]^{|\setu|}\bigg) \\
  &\qquad\,=\, \frac{\tau}{n}
  \sum_{\setw\subseteq\{s+1:d\}}
  \bigg(\sum_{s\in\setu\subseteq\{1:s\}} \gamma_{\setu\cup\setw}^\lambda\, [2\zeta(\alpha\lambda)]^{|\setu\cup\setw|}\bigg)
  \bigg(\sum_{\setu\subseteq\{1:s\}} \gamma_{\setu\cup\setw}^\lambda\, [2\zeta(\alpha\lambda)]^{|\setu\cup\setw|}\bigg) \\
  &\qquad\,\le\, \frac{\tau}{n}
  \bigg(\!\max_{\setw\subseteq\{s+1:d\}}\!\!
  \sum_{s\in\setu\subseteq\{1:s\}} \!\!\gamma_{\setu\cup\setw}^\lambda\, [2\zeta(\alpha\lambda)]^{|\setu\cup\setw|}\!\bigg)
  \bigg(\!\sum_{\setw\subseteq\{s+1:d\}}\!
  \sum_{\setu\subseteq\{1:s\}} \!\!\gamma_{\setu\cup\setw}^\lambda\, [2\zeta(\alpha\lambda)]^{|\setu\cup\setw|}\!\bigg) \\
  &\qquad\,=\, \frac{\tau}{n}
  \bigg(\max_{\setw\subseteq\{s+1:d\}}
  \sum_{s\in\setu\subseteq\{1:s\}} \gamma_{\setu\cup\setw}^\lambda\, [2\zeta(\alpha\lambda)]^{|\setu\cup\setw|}\bigg)
  \bigg(\sum_{\setu\subseteq\{1:d\}}
  \gamma_{\setu}^\lambda\, [2\zeta(\alpha\lambda)]^{|\setu|}\bigg).
\end{align*}
This leads to
\begin{align} \label{eq:max}
  S_{d}(\bsz^*)
  &\,\le\, \sum_{s=1}^{d}
  \bigg[ \frac{\tau}{n}
  \bigg(\max_{\setw\subseteq\{s+1:d\}}
  \sum_{s\in\setu\subseteq\{1:s\}} \gamma_{\setu\cup\setw}^\lambda\, [2\zeta(\alpha\lambda)]^{|\setu\cup\setw|}\bigg)
  \bigg(\sum_{\setu\subseteq\{1:d\}}
  \gamma_{\setu}^\lambda\, [2\zeta(\alpha\lambda)]^{|\setu|}\bigg) \bigg]^{1/\lambda} \nonumber\\
  &\,\le\,
  \bigg[ \frac{\tau}{n}
  \bigg(\sum_{s=1}^{d} \max_{\setw\subseteq\{s+1:d\}}
  \sum_{s\in\setu\subseteq\{1:s\}} \gamma_{\setu\cup\setw}^\lambda\, [2\zeta(\alpha\lambda)]^{|\setu\cup\setw|}\bigg)
  \bigg(\sum_{\setu\subseteq\{1:d\}}
  \gamma_{\setu}^\lambda\, [2\zeta(\alpha\lambda)]^{|\setu|}\bigg) \bigg]^{1/\lambda}. \nonumber\\
\end{align}

We remark at this point that, unlike a typical induction proof for a CBC
construction, there is no induction in this proof.

We consider two ways to proceed. The first way is to replace the maximum
in~\eqref{eq:max} by the sum, which yields
\begin{align*}
  S_{d}(\bsz^*)
  &\,\le\,
  \bigg[ \frac{\tau}{n} \bigg(\sum_{s=1}^{d}
  \sum_{s\in\setu\subseteq\{1:d\}} \gamma_{\setu}^\lambda\, [2\zeta(\alpha\lambda)]^{|\setu|}\bigg)
  \bigg(\sum_{\setu\subseteq\{1:d\}} \gamma_\setu^\lambda\, [2\zeta(\alpha\lambda)]^{|\setu|}\bigg) \bigg]^{1/\lambda} \\
  &\,=\,
  \bigg[ \frac{\tau}{n} \bigg(
  \sum_{\emptyset\ne\setu\subseteq\{1:d\}} |\setu|\,\gamma_{\setu}^\lambda\, [2\zeta(\alpha\lambda)]^{|\setu|}\bigg)
  \bigg(\sum_{\setu\subseteq\{1:d\}} \gamma_\setu^\lambda\, [2\zeta(\alpha\lambda)]^{|\setu|}\bigg) \bigg]^{1/\lambda}.
\end{align*}
The second way is to apply the assumption \eqref{eq:decay} in
\eqref{eq:max}, so that the maximum drops out to yield
\begin{align*}
  S_d(\bsz^*)
  &\,\le\,
  \bigg[ \frac{\tau\,\xi}{n} \bigg(\sum_{s=1}^{d}
  \sum_{s\in\setu\subseteq\{1:s\}} \gamma_{\setu}^\lambda\, [2\zeta(\alpha\lambda)]^{|\setu|}\bigg)
  \bigg(\sum_{\setu\subseteq\{1:d\}} \gamma_\setu^\lambda\, [2\zeta(\alpha\lambda)]^{|\setu|}\bigg) \bigg]^{1/\lambda} \\
  &\,=\,
  \bigg[ \frac{\tau\,\xi}{n} \bigg(
  \sum_{\emptyset\ne\setu\subseteq\{1:d\}} \gamma_{\setu}^\lambda\, [2\zeta(\alpha\lambda)]^{|\setu|}\bigg)
  \bigg(\sum_{\setu\subseteq\{1:d\}} \gamma_\setu^\lambda\, [2\zeta(\alpha\lambda)]^{|\setu|}\bigg) \bigg]^{1/\lambda},
\end{align*}
which does not contain the factor $|\setu|$ inside the first sum.
\end{proof}

We summarize the main conclusion of this paper in the following theorem,
which states that we achieve the best possible convergence rate for
lattice algorithms as shown in \cite{BKUV17}.

\begin{theorem} \label{thm:final}
Given $d\ge 1$, $\alpha>1$ and weights
$\{\gamma_\setu\}_{\setu\subset\bbN}$, let $n$ be prime and $M>0$. The
lattice algorithm \eqref{eq:Af}, with index set \eqref{eq:AdM} and
generating vector $\bsz$ obtained from the CBC construction following
Algorithm~\ref{alg}, satisfies for all $\lambda\in (\frac{1}{\alpha},1]$,
\begin{align*}
  e^{\rm wor\mbox{-}app}_{n,d,M}(\bsz)
  &\,\le\, \bigg(\frac{1}{M} + M\, S_d(\bsz) \bigg)^{1/2} \\
  &\,\le\, \Bigg(\frac{1}{M} + M \bigg[\frac{\tau}{n} \bigg(
  \sum_{\emptyset\ne\setu\subseteq\{1:d\}} |\setu|\,\gamma_{\setu}^\lambda\, [2\zeta(\alpha\lambda)]^{|\setu|}\bigg)
  \bigg(\sum_{\setu\subseteq\{1:d\}} \gamma_\setu^\lambda\, [2\zeta(\alpha\lambda)]^{|\setu|}\bigg)\bigg]^{1/\lambda}
 \Bigg)^{1/2},
\end{align*}
where $\tau = \max(6,2.5+2^{2\alpha\lambda+1})$.

Taking $M = n^{1/(2\lambda)}$, we obtain a simplified upper bound
\begin{align*}
  e^{\rm wor\mbox{-}app}_{n,d,M}(\bsz)
  \,\le\, \frac{\sqrt{2}\,\tau^{1/(2\lambda)}}{n^{1/(4\lambda)}} \bigg(
  \sum_{\setu\subseteq\{1:d\}} \max(|\setu|,1)\,\gamma_{\setu}^\lambda\, [2\zeta(\alpha\lambda)]^{|\setu|}\bigg)^{1/\lambda}.
\end{align*}
Hence
\[
  e^{\rm wor\mbox{-}app}_{n,d,M}(\bsz) \,=\, \calO(n^{-\alpha/4 + \delta}), \quad\delta>0,
\]
where the implied constant is independent of $d$ provided that
\[
  \sum_{\setu\subset\bbN,\,|\setu|<\infty} \max(|\setu|,1)\,\gamma_{\setu}^{\frac{1}{\alpha-4\delta}}\,
  [2\zeta\big(\tfrac{\alpha}{\alpha-4\delta}\big)]^{|\setu|}
  \,<\, \infty.
\]

If the weights satisfy \eqref{eq:decay} for some $\xi\ge1$ then the
$|\setu|$ factor inside the sums can be replaced by $1$ as long as $\tau$
is replaced by $\tau\,\xi$.
\end{theorem}

\begin{proof}
The theorem is a consequence of combining \eqref{eq:bal}, \eqref{eq:Sd},
\eqref{eq:final} and then balancing the terms by choosing $M$ in relation
to $n$ according to \eqref{eq:chooseM} and then taking $\lambda =
1/(\alpha-4\delta)$.
\end{proof}

As a closing remark we note that $\max(|\setu|,1)\le (e^{1/e})^{|\setu|} =
(1.4446\cdots)^{|\setu|}$. This means that the constant is independent of
$d$ if
\[
  \sum_{\setu\subset\bbN,\,|\setu|<\infty} \gamma_{\setu}^{\frac{1}{\alpha-4\delta}}\,
  [2e^{1/e} \, \zeta\big(\tfrac{\alpha}{\alpha-4\delta}\big)]^{|\setu|}
  \,<\, \infty.
\]
This condition is slightly more demanding, but easier on the eyes, and
suggests that the factor of $|\setu|$ which popped up in the estimates is
not really worse than some of the other estimates which were already made
on the way.

\paragraph{Acknowledgements}
We gratefully acknowledge the financial support from the Australian
Research Council (DP180101356).

\bibliographystyle{plain}

\end{document}